\documentclass{article}

\usepackage{latexsym}
\usepackage{amsthm}
\usepackage{amssymb}
\usepackage{amsmath}
\usepackage{amsfonts}
\usepackage{mathrsfs}
\usepackage[all]{xy}
\usepackage{texdraw}
\usepackage{graphicx}
\usepackage{psfrag}
\usepackage{makeidx}
\usepackage{color}
\usepackage{fullpage}
\usepackage[shortlabels]{enumitem}
\usepackage{tikz}
\usetikzlibrary{matrix}

\theoremstyle{plain}
\newtheorem{theorem}{Theorem}

\newtheorem{lemma}[theorem]{Lemma}

\newtheorem{remark}[theorem]{Remark}
\newtheorem{condition}{Condition}
\newtheorem{corollary}[theorem]{Corollary}
\newtheorem{question}{Question}

\theoremstyle{definition}

\newcommand{\cstar}[1]{C^*_\lambda(#1)}
\newcommand{\Z}{\mathbb{Z}}
\newcommand{\N}{\mathbb{N}}
\newcommand{\Q}{\mathbb{Q}}

\begin{document}

\title{Ideal structure of the $C^*$-algebra of Thompson group $T$}

\author{Collin Bleak, Kate Juschenko}
%
%
\date{\today}

\maketitle
\abstract{In a recent paper Uffe Haagerup and Kristian Knudsen Olesen  show that for Richard Thompson's group $T$,  if 
there exists  a finite set $H$ which can be decomposed as  disjoint union of sets $H_1$ and $H_2$ with $\sum_{g\in H_1}\pi(g)=\sum_{h\in H_2}\pi(h)$ and such that the closed ideal generated by $\sum_{g\in H_1}\lambda(g)-\sum_{h\in H_2}\lambda(h)$ coincides with $C^*_\lambda(T)$, then the Richard Thompson group $F$ is not amenable. In particular, if $\cstar{T}$ is simple then $F$ is not amenable. Here we prove the converse, namely, if $F$ is not amenable then we can find two sets $H_1$ and $H_2$ with the above properties.  The only currently available tool for proving simplicity of group $C^*$-algebra is Power's condition. We show that it fails for $\cstar{T}$ and present an apparent weakening of that condition which could potentially  be used for various new groups $H$ to show the simplicity of $\cstar{H}$.  While we use our weakening in the proof of the first result, we also show that the new condition is still too strong to be used to show the simplicity of $\cstar{T}$.  Along the way, we give a new application of the Ping-Pong Lemma to find  free groups as subgroups in groups of homeomorphisms of the circle generated by elements with rational rotation number. 
\section{Introduction}
There has been a long-standing interest in the question of the amenability of Richard Thompson's group $F$, introduced in Thompson's notes of 1965 (see the survey \cite{CFP} for a general background on the three Thompson groups $F<T<V$), with many failed attempts to prove either the amenability or non-amenability of $F$.  The groups $F<T<V$ arise in many areas of mathematics for reasons which are not entirely understood.  One plausible explanation is that they express in some fundamental way connections through Category Theory with associativity and versions of commutativity (see \cite{BrinCat, GeoGuz, DehGeo} for some discussion of these connections), which of course are fundamental aspects of any theory involving products. Regardless of the cause, it is still the case that these groups arise naturally in many areas of mathematics including dynamics, logic, topology, and more obviously geometric group theory.  One fetching example of such an appearance is in e.g., the relationship between the group $F$ and the theory of associahedra, and in particular, the theory related to the the proof of the Four Colour Theorem \cite{BowlinBrin}.  In any case, it is well known now that the Richard Thompson groups are fundamental.   In this paper, we will be investigating some structures related to $T$ which have implications towards the amenability of $F$, by exploring some related questions from the theory of $C^*$ algebras.

Since the foundational paper  \cite{PowersFreeAlgebraSimplicity} of Powers, there has been a long-standing interest in whether there could exist a group $G$ with unique trace, for which its reduced $C^*$ algebra $\cstar{G}$ is not simple.  The paper  \cite{PowersFreeAlgebraSimplicity}  contains various conditions on a group which imply uniqueness of trace and/or simplicity of the algebra.  (See Section \ref{tests} for a discussion of an often-used condition of Powers.) The question of the equivalence of these two properties is explicitly stated in the papers \cite{dlH}, \cite{dlH2} of Pierre de la Harpe, and his excellent overview of the problem \cite{dlH3}.  For more on the general problem we refer the reader to the paper \cite{Robin} of Robin Tucker-Drob.

  Uffe Haagerup and Kristian Knudsen Olesen in  \cite{HaagerupOlesen} show that the simplicity of $\cstar{T}$ implies the non-amenability of $F$ via a construction given below (see Theorem \ref{HO-IFF} below), while it is well known that $T$ is a group with unique trace.  Thus, in this paper we begin to investigate the ideal structure of $\cstar{T}$.
 
 Haagerup and Olesen's idea showing that that the simplicity of the algebra $\cstar{T}$ implies the non-amenability of $F$  runs as follows.  Consider $T$ ``acting'' on the interval $[0,1]$.  Assume that the stabiliser of $0$, which is the standard copy of the Thompson group $F$ in $T$, is amenable.  Since the action of $T$ on $\mathbb{Z}[\frac{1}{2}]$ is transitive, we have that the representation induced by this action, $\pi:T\rightarrow B(l_2(\mathbb{Z}[\frac{1}{2}]))$, is weakly contained in the left regular representation. From this one sees that there is a unique $*$-homomorphism from $\cstar{T}$ into the $C^*$-algebra generated by $\pi(T)$.  
 
Consider now a finite subset $H$ of $T$ so that  $H=H_1\sqcup H_2$ and with

  \[
  \sum_{g\in H_1}\pi(g)-\sum_{h\in H_2}\pi(h)=0.
  \]

The simplicity of $\cstar{T}$ now implies that the ideal generated by  $\sum_{g\in H_1}\lambda(g)-\sum_{h\in H_2}\lambda(g)$ is proper.  However, this is not possible since $\pi$ is non-trivial, so $F$ must be non-amenable.

Our main result is to show a (partial) converse of the above program (we except the simplicity of $\cstar{T}$).

\begin{theorem}\label{HO-IFF}
Thompson's group F is non-amenable if and only if there exists  a finite set $H$ which can be decomposed as  disjoint union of sets $H_1$ and $H_2$ with $\sum_{g\in H_1}\pi(g)=\sum_{h\in H_2}\pi(h)$ and such that the closed ideal generated by $\sum_{g\in H_1}\lambda(g)-\sum_{h\in H_2}\lambda(h)$ coincides with $C^*_\lambda(T)$.
\end{theorem}

The rest of this article proceeds as follows. 

In Section \ref{ideals} we present the proof of the main result, modulo our Lemma \ref{fixed-point}.  

In the Section \ref{tests} we discuss a theorem of Powers and another of Kesten which together produce Condition \ref{kesten-test}, a main test which is often used to detect the simplicity of $\cstar{G}$ for a group $G$.  We show that Condition \ref{kesten-test} fails to apply to Thompson's group $T,$  and we offer a new test (Condition \ref{weak-kesten}) which appears to be a natural weakening of Condition \ref{kesten-test}, and which may be of use for various groups $G$ to show that $\cstar{G}$ is simple. Like Condition \ref{kesten-test}, the new Condition \ref{weak-kesten} gives the simplicity of $\cstar{G}$ for a group $G$ if for all finite subsets $H$ of $G\,\backslash\!\left\{e\right\}$, we can carry out a certain construction (creating large free subgroups in a certain way from $H$). We also show here that Condition \ref{weak-kesten} is still not weak enough to show that the algebra $\cstar{T}$ is simple, while giving some evidence that Condition \ref{weak-kesten} is properly weaker than the commonly used Condition \ref{kesten-test}. 

The authors are unaware of any group with unique trace, for which Condition \ref{kesten-test} fails.  We believe $T$ is a first example of such a group.

  In Section \ref{p-pong} we provide a short discussion of the historical Ping-Pong Lemma and we prove a version of the Ping-Pong Lemma (Lemma \ref{free-powering-one}) useful for detecting free subgroups when considering a group generated by two homeomorphisms of $S^1$ with rational rotation numbers.  Lemma \ref{free-powering-one} is an essential ingredient in the  proof of Lemma \ref{fixed-point}.  
  
  In Section \ref{trying-conditions}, we give proofs that we can carry out the construction of Condition \ref{weak-kesten} for many cases of finite $H\subset T\,\backslash\!\left\{e\right\}$, including the cases we need to prove Lemma \ref{fixed-point}.  We also describe some cases of $H$ where we cannot carry out the construction of Condition \ref{weak-kesten}, but where related constructions do produce large free subgroups.  
  
  In our final Section \ref{conjectures} we state some remaining questions which we find interesting.
  
{\it Acknowledgments:} The authors are grateful to Uffe Haagerup and Kristian Knudsen Olesen for sharing the early drafts of their work \cite{HaagerupOlesen} with us, and to Martin Kassabov and Justin Tatch Moore for feedback on this article while it was under construction.
\section{Non-amenability of $F$ and a condition on ideals of its $C^*$-algebra.}\label{ideals}

\begin{theorem}
The Thompson group F is not amenable if and only if there exists  a finite set $H$ which can be decomposed as  disjoint union of sets $H_1$ and $H_2$ with $\sum_{g\in H_1}\pi(g)=\sum_{h\in H_2}\pi(h)$
and such that the closed ideal generated by $\sum_{g\in H_1}\lambda(g)-\sum_{h\in H_2}\lambda(h)$ coincides with $C^*_\lambda(T)$.
\end{theorem}
\begin{proof}

One part of the theorem follows from the draft \cite{HaagerupOlesen} of Haagerup and Olesen. \\

It is left to show that if $F$ is not amenable then there is a finite set $H$ which satisfies the conditions of the theorem.  Let $x_1$ and $x_2$ be the standard generators of the copy of Thompson's group $F$ in $T$, supported in $[\frac{1}{2},1]\subset S^1$, with supports in $[\frac{1}{2},1]$ and $[\frac{3}{4},1]$ respectively. Let $g_1=x_1^r$ and $g_2=x_2^r$ be the conjugated copies of these generators, where $r$ is rotation by $1/2$, so that $g_1$ and $g_2$ act on the interval $[0,\frac{1}{2}]\subset S^1$ and generate a copy of $F$ there, with trivial action on $[\frac{1}{2},1]=[\frac{1}{2},0]$. Define $H=\{e, g_1, g_2, x_1, x_2, g_1x_1,g_2x_2 \}$. Define $H_1=\{g_1, g_2, x_1,x_2\}$ and $H_2$ to be the rest of the set $H$. Obviously, $\sum_{g\in H_1}\pi(g)-\sum_{h\in H_2}\pi(h)=0$. 

Let us show  that the ideal, $J$, generated by $\sum_{g\in H_1}\lambda(g)-\sum_{h\in H_2}\lambda(h)$ is the whole reduced $C^*$-algebra of $T$. Note that $\|1+\lambda(g_1)+\lambda(g_2)\|\leq C<3$ by assumption. Moreover, the point $p=\frac{7}{8}$ is not  fixed  by any element from  the set $E=\{ x_1, x_2, g_1x_1,g_2x_2  \}$. Thus we can apply Lemma \ref{fixed-point} for the set $E$: let $\varepsilon>0$ and let $g$ and $c_1,\ldots, c_n\in C_T(g)$ be such that 
$$\| \sum_{s\in E} \sum_{i=1}^n\lambda((sg)^{c_i}) \|\leq \varepsilon n.$$

Note that the element $b=\lambda(g-\frac{1}{3}[e+g_1+g_2]g+\frac{1}{3} \sum_{s\in E} sg)$ is in $J$, thus $\frac{1}{n}\sum_{i=1}^n\lambda(c_i)b\lambda({c_i}^{-1}) \in J$. The distance between the element $\frac{1}{n}\sum_{i=1}^n\lambda(c_i)b\lambda({c_i}^{-1})$ and $\lambda(g)$ is strictly smaller then $1$ for large $n$. Indeed,

\begin{align*}
\|\lambda(g)-b\|\leq& \frac{1}{3}\|1+\lambda(g_1)+\lambda(g_2)\| + \frac{1}{3n}\|\sum_{s\in E} \sum_{i=1}^n\lambda((sg)^{c_i})  \|\leq C+\varepsilon.
\end{align*}
thus we have found an invertible element in $J$, therefore $J=C^*_\lambda(T)$.
\end{proof}

\section{Powers' test}\label{tests}
In \cite{PowersFreeAlgebraSimplicity} Powers' gives the following test for the simplicity of the algebra $\cstar{G}$ over a group $G$.

\begin{theorem}\label{sumTest}
  If for all non-empty $H\subset G$ with $|H|<\infty$, $e\not\in H$ and for all positive integers $n$ there is a set $\{c_1,c_2,\ldots,c_n\}\subset G$ so that 
\[
\lim_{n\to\infty}\frac{1}{n}||\Sigma_{i=1}^n \lambda(c_ih{c_i}^{-1})||=0,\,\forall h\in H,
\]
then $\cstar{G}$ is simple.
\end{theorem}

Let $G$ be a group generated by a finite set $S$ with $S=S^{-1}$, then $\frac{1}{|S|}||\Sigma_{h\in S} \lambda(h)||$ is equal to the spectral radius of the simple random walk on the Cayley graph of $G$ with respect to $S$, denoted by $\rho(G,S)$. The spectral radius of the simple random walk have been computed for many groups. Kesten, \cite{kesten}, showed that if $S=\{g_1,\ldots,g_n\}$ is {\it a free set}, i.e., $g_1,\ldots,g_n$ are standard generators of the free group of rank $n$, then the spectral radius is
$$\rho(G,S)=\frac{\sqrt{2n-1}}{n}$$

Thus the following condition implies the hypothesis of the Theorem \ref{sumTest}.
\begin{condition}\label{kesten-test}
For all finite subsets $H\subset G$ with $e\not\in H$ and for all positive integers $n$ there is a set $\{c_1,c_2,\ldots,c_n\}\subset G$ so that 
\[
\langle h^{c_1},h^{c_2}, \ldots, h^{c_n}\rangle
\]
is a free subgroup of $G$ of rank $n$ for all $h\in H$.
\end{condition}

If $g$ is a bijection from a set $X$ to itself, denote by $Supp(g):=\left\{x\in X\mid g\cdot x\neq x\right\}$ and $Fix(g):=X\backslash Supp(g)$, the support and the set of points fixed by  $g$, respectively.

The following remark holds true for groups of permutations of a set $X$.
\begin{remark}\label{disjointSupportsAndConjugation}
Let $X$ be a set, and $G$ the group of bijections from $X$ to $X$.  
Suppose $h_1$, $h_2\in G\backslash\left\{1\right\}$ so that $Supp(h_1)\cap Supp(h_2)=\emptyset$.  If $c_1$, $c_2\in G$ so that $Supp(h_1^{c_1})\cup Supp( h_1^{c_2})=X$ then $Supp(h_2^{c_1})\cap Supp(h_2^{c_2})=\emptyset$.
\end{remark}

\begin{proof}
Suppose 
\[
X=Supp(h_1^{c_1})\cup Supp(h_1^{c_2}) (= c_1\cdot Supp(h_1)\cup c_2\cdot Supp(h_1)).
\]
If there is $x\in X$ so that  $x\in Supp(h_2^{c_1})\cap Supp(h_2^{c_2})$, then $x=c_1\cdot y$ and $x=c_2\cdot z$, where $y$ and $z$ are in $Fix(h_1)$.  In particular, $x\in c_1\cdot Fix(h_1)\cap c_2\cdot Fix(h_1)$.  This implies that $c_1\cdot Supp(h_1)\cup c_2\cdot Supp(h_1)\neq X$.
\end{proof}

Remark \ref{disjointSupportsAndConjugation} immediately implies that we cannot  use Condition \ref{kesten-test} when approaching the question of the simplicity of the algebra $\cstar{T}$.

\begin{corollary}\label{negative}
Suppose that $H\subset T$ admits elements $h_1$ and $h_2$ so that $Supp(h_1)\cap Supp(h_2)=\emptyset $.  Then for $n\geq 2$ there is no set of elements $\left\{c_1,c_2\ldots,c_n\right\}$ so that $\langle h^{c_1},h^{c_2}, \ldots,h^{c_n}\rangle$ is a free group on $n$ generators for all $h\in H$.
\end{corollary}
\begin{proof}
Suppose $H:=\left\{h_1,h_2,\ldots,h_k\right\}$ is a finite set with cardinality at least two, and $h_1$ and $h_2$ are in $H$ so that $Supp(h_1)\cap Supp(h_2)=\emptyset.$  Further suppose that $n\geq 2$ is fixed and $c_1$, $c_2$,$\ldots$, $c_n$ are chosen so that for all $h\in H$, we have $\langle h^{c_1},h^{c_2},\ldots, h^{c_n}\rangle$ is free on $n$ generators.  As proven and Brin and Squier's paper \cite{BSPLR}, the group of piecewise linear homeomorphisms of the unit interval has no non-abelian free subgroups, so we see immediately that $Supp(h_1^{c_1})\cup Supp(h_1^{c_2})=S^1$.  Now by Remark \ref{disjointSupportsAndConjugation} we know that  $Supp(h_2^{c_1})\cap Supp(h_2^{c_2})=\emptyset$  Therefore $\langle h_2^{c_1},h_2^{c_2}\rangle\cong \Z\times\Z$.
\end{proof}

We now offer an apparently weaker version of Condition \ref{kesten-test} which will be used throughout the remainder of this article.  First, we need a supporting theorem.

Below, let $C_G(g)$ be the centralizer of an element $g$ in $G$.

\begin{theorem}
Let $H\subset G$ be a finite set and there is an element $w\in H$ such that for all positive integers $n$ there is a set $\{c_1,c_2,\ldots,c_n\}\subset G$ and $r,s\in G$ such that $c_i \in C_G(swr)$ for all $i$ and 
\[
\lim_{n\to\infty}\frac{1}{n}||\Sigma_{i=1}^n \lambda(c_isgr{c_i}^{-1})||=0,\,\text{ for all } g\in H\backslash \{w\},
\]
then for all coefficients $\beta_g$ indexed by $H$ with $\beta_w\neq 0$, the ideal generated by $\sum_{g\in H}\beta_g \lambda(g)$ is equal to  $\cstar{G}$.

\end{theorem}
\begin{proof}
Let $I$ be an ideal in $C^*_{\lambda}(G)$ generated by $b:=\sum_{g\in H}\beta_g \lambda(g)$. Assume that $I$ is proper. The closure of $I$ is proper, thus we can assume $I$ is closed. Note that $\Sigma_{i=1}^n \lambda(c_is)b\lambda(rc_i^{-1})\in I$. Since $c_i \in C_G(swr)$ we have 
\begin{align*}
\|\lambda (swr)-\frac{1}{\beta_w n}\Sigma_{i=1}^n \lambda(c_is)b\lambda(rc_i^{-1})\|=&\frac{1}{\beta_w n} \|\Sigma_{g\in H\backslash \{w\}}\Sigma_{i=1}^n \beta_g\lambda(c_isgrc_i^{-1})\| \\
\leq& \frac{1}{\beta_w} \Sigma_{g\in H\backslash \{w\}}|\beta_g|\frac{1}{n}\|\Sigma_{i=1}^n \lambda(c_isgrc_i^{-1})\|\\
=&\max(|\beta_g|/ \beta_w:g\in H)\cdot \max(\frac{1}{n}\|\Sigma_{i=1}^n \lambda(c_isgrc_i^{-1})\|:g\in H).
\end{align*}
By our assumptions, the last quantity can be arbitrarily small for large $n$. Thus there is an element in $I$ which is on distance less then $1$ to a unitary operator, this implies that it is invertible and $I=\cstar{G}$.
\end{proof}

Applying the theorem above  to the set $H\cup \{e\}$ shows that the following condition implies simplicity of $C^*_{\lambda}(G)$ :

\begin{condition}\label{weak-kesten}
For all finite non-empty subsets $H\subset G$, $e\not\in H$ and for all positive integers $n$ there are $r$, $s\in G$ and a set $\{c_1,c_2,\ldots,c_n\}\subset C_G(sr)$  such that the set $\{c_k(sgr)c_k^{-1}: k=1,\ldots,n\}$ is free for all $g\in H$.
\end{condition}

Condition \ref{kesten-test} implies Condition \ref{weak-kesten} and it seems that the other implication is false.  However, Condition \ref{weak-kesten} is still inadequate for showing that $\cstar{T}$ is simple.
\begin{lemma}
There are $g_1$, $g_2\in T\backslash \{e\}$ so that for any $r$, $s\in T$ there are no elements $c_1$, $c_2$, $c_3$, and $c_4\in C_T(sr)$ with both $G_1=\langle (sg_1r)^{c_1},(sg_1r)^{c_2}, (sg_1r)^{c_3},(sg_1r)^{c_4}\rangle$ and $G_2=\langle (sg_2r)^{c_1},(sg_2r)^{c_2}, (sg_2r)^{c_3},(sg_2r)^{c_4}\rangle$ free on four generators.
\end{lemma}
\begin{proof}
Let $g_1$, $g_2\in T$ so that $Supp(g_1)=(0,1/2)$ and $Supp(g_2)=(1/2,1)$.  Let $r$ and $s\in T$ and suppose $c_1$, $c_2$, $c_3$ and $c_4\in C_T(sr)$.  Set $k_{ij}=(sg_ir)^{c_j}$, for $i$, $j\in \{1,2,3,4\}$, and suppose that  $c_1$, $c_2$, $c_3$ and $c_4$ were so chosen so that $G_i=\langle k_{i1},k_{i2}, k_{i3}, k_{i4}\rangle$ is free on four generators for $1\leq i\leq 2$.

Consider the intervals $X_{i1}=(c_1r^{-1})\cdot Fix(g_i)$, $X_{i2}=(c_2r^{-1})\cdot Fix(g_i)$, $X_{i3}=(c_3r^{-1})\cdot Fix(g_i)$, and $X_{i4}=(c_4r^{-1})\cdot Fix(g_i)$.  If $x_{ij}\in X_{ij}$, then $k_{ij}\cdot x_{ij} = c_jsg_irc_j^{-1}\cdot x_{ij} =(c_jsg_i)\cdot y_{ij}=(c_js)\cdot y_{ij} = (c_jsrc_j^{-1})\cdot x_{ij} = (sr)\cdot x_{ij}$, as for all $i$, $j$ we have $y_{ij}\in Fix(g_i)$.  That is, $k_{ij}$ acts as $sr$ over $X_{ij}$.

Further, consider the elements $f_{i,ab}= k_{ia}^{-1}k_{ib}$, where $i\in \{1,2\}$ and $a\neq b \in \{1,2,3,4\}$.  It is immediate that $\langle f_{i,ab},f_{i,cd}\rangle$ is free on two generators if either $b\neq c$ or $d\neq a$.  Therefore, by Brin and Squier's result (from \cite{BSPLR}) that $PL_o(I)$ has no non-abelian free subgroups, we know that $Fix(f_{i,ab})\cap Fix(f_{i,cd})=\emptyset$ for $i\in\{1,2\}$ and either $b\neq c$ or $d\neq a$.  Now, for instance, if there is an index $i$ and some point $p\in X_{i1}\cap X_{i2}\cap X_{i3}$,  then both $f_{i,12}$ and $f_{i,13}$ must fix $p$, which is a contradiction.  Therefore we see that $X_{i1}$, $X_{i2}$, and $X_{i3}$ cannot share a common point for any index $i$.  By the same argument, for any valid indices $i$, $a$, $b$, and $c$ (where $i\in \{1,2\}$ and $a\neq b\neq c\neq a$) we see that $X_{ia}\cap X_{ib}\cap X_{ic}=\emptyset$.

One now sees immediately that for any valid indices $i$, $a$,$b$, and $c$ (where $i\in \{1,2\}$ and $a\neq b\neq c\neq a$) we must also have that $X_{ia}\cup X_{ib}\cup X_{ic}=S^1$.  This follows as otherwise there is some point $p$ in the intersection $X_{ja}\cap X_{jb}\cap X_{jc}$ for the index $j\neq i$ (since $X_{1*} = \overline{S^1\backslash X_{2*}}$ for any index $*$).  

Suppose that for some indices $i$, $a\neq b$ we have that $X_{ia}\subset X_{ib}$, and let $c$ and $d$ be the two remaining distinct indices of $\{1,2,3,4\}\backslash\{a,b\}$.  Let $p$ be an endpoint of $X_{ib}$.  We have that $p$ must be in both $X_{ic}$ and $X_{id}$, otherwise their will be some point $q \in S^1\backslash X_{ib}$ which is near to $p$ so that $q$ is not in either of $X_{ia}\cup X_{ib}\cup X_{ic}= X_{ib}\cup X_{ic}$ or $X_{ia}\cup X_{ib}\cup X_{id}= X_{ib}\cup X_{id}$.  But this contradicts the fact that $X_{ib}\cap X_{ic}\cap X_{id}= \emptyset$.

It now immediately follows that for any index $i$ and two distinct indices $a$ and $b$, we have that $X_{ia}\cap X_{ib}$ is a non-empty closed interval (possibly a single point) while $X_{ia}\cup X_{ib}$ is also a closed interval which misses some points in $S^1$.

But now we are done as follows.  For any index $i$ the intervals $X_{i1}$ $X_{i2}$ and $X_{i3}$ cover the circle, and have the properties that each pair of sets intersects in an interval, and  no pair covers the whole circle.  Now consider $X_{i4}$.  It must likewise intersect both $X_{i1}$ and $X_{i2}$ non-trivially, and the union of $X_{i1}$, $X_{i2}$ and $X_{i4}$ also covers the whole circle.  Therefore the end of $X_{i1}$ which is not in $X_{i2}$ is in both $X_{i3}$ and $X_{i4}$.  Hence $X_{i1}\cap X_{i3}\cap X_{i4}\neq \emptyset$, which implies that the group $G_i$ cannot be free on four generators, as $f_{i,13}$ and $f_{i,14}$ share a common fixed point and will not generate a free subgroup of $G_i$.
\end{proof}

\begin{remark}We observe that it is still plausible that even with $g_1$ and $g_2$ as in the proof above (supports over $(0,1/2)$ and $(1/2,1)$, respectively), one could plausibly find $r$, $s$, and $c_1$, $c_2$, and $c_3\in C_T(sr)$ so that setting $k_{ij}= c_jsg_irc_j^{-1}$ as above we would have $H_r=\langle k_{r1}, k_{r2}, k_{r3}\rangle$ free on three generators for both $r=1$ and $r=2$, where the related claim for even two generator free groups could not be conceived of under Condition \ref{kesten-test}.
 \end{remark}
 
\section{A Ping-Pong Lemma for orientation preserving homeomorphisms of $S^1$}\label{p-pong}

In this section, we prove a version of the Ping-Pong Lemma which we are using in our main argument.  In the notations below we write all actions as left actions, in keeping with the tradition in the $C^*$ literature, although much Thompson groups literature uses right action. In particular, if $x\in S^1$ and $s$,$t\in T$, we write $tx$ for the image of $x$ under $t$, and the conjugation $s^t:= sts^{-1}$, which means, apply $s^{-1}$ first, then $t$, and then $s$. We consider finite sets with repetitions. 

In support of that lemma we ask the reader to recall an ordinary statement of Fricke and Klein's Ping-Pong Lemma (first proven in \cite{FrickeKlein}, but we give a different statement), and two further facts, one quite classical.

\begin{lemma}(Ping-Pong Lemma)\label{ping-pong} Let $G$ be a group of permutations on a set $X$, and let $a$, $b\in G$, where $b^2\neq 1$.  If $X_a$ and $X_b$ are two subsets of $X$ so that neither is contained in the other, and for all integers $n$ we have $b^n\cdot X_a\subset X_b$ whenever $b^n\neq 1$, and $a^n\cdot X_b\subset X_a$ whenever $a^n\neq 1$, then $\langle a,b\rangle$ factors naturally as the free product of $\langle a\rangle$ and $\langle b\rangle$.   In particular, $\langle a,b\rangle\cong \langle a\rangle *\langle b\rangle$.
\end{lemma}

Suppose that $f:S^1\rightarrow S^1$ is an orientation preserving homeomorphism of the circle $S^1=\mathbb{R}/\mathbb{Z}$, then $f$ may be lifted to a homeomorphism of $\mathbb{R}$ by $F(x+m)=F(x)+m$ for every  $x$ and $m$. The rotation number of $f$ is defined to be $Rot(f)=\lim_{n\rightarrow \infty} (F^n(x)-x)/{n}$.
The following theorem is generally relevant to the arguments in the final section of this paper, and appears first in \cite{GhysSergiescu}, although there now exist many different proofs, the shortest of which appears to be in \cite{BKMStructure}.
\begin{theorem}
Every element of Thompson's group $T$ has rational rotation number.
\end{theorem}
The last tool we need in order to establish our own version of the Ping Pong lemma is the following classical result of Poincar\`e.
\begin{lemma}(Poincar\`es Lemma, circa 1905)
If $f$ is an orientation preserving homeomorphism of $S^1$ and $f$ has rotation number $p/q$ in lowest terms, then there is an orbit in $S^1$ of size exactly $q$ under the action of $\langle f\rangle$.
\end{lemma}

We are now in a good position to quote and prove our main technical tool.

\begin{lemma}\label{free-powering-one}
Suppose $a$ and $b$ are orientation preserving homeomorphisms of the circle $S^1$ with rational rotation numbers $Rot(a)=p/q$ and $Rot(b)=r/s$ in lowest non-negative terms where
\begin{enumerate}
\item  $b$ is not torsion, and 
\item  if $x\in Fix(b^s)$ and $j\in\Z$ with $a^j\neq 1_T$, then we have $a^jx\not\in Fix(b^s)$,
\end{enumerate}
then, there is a positive integer $k$ so that $a$ and $b^k$ are a free basis for the group $\langle a, b^k\rangle$.
\end{lemma}
\begin{proof}
In the proof below, let us take $a$, $b\in T$ and $p$, $q$, $r$, $s\in \N$ as in the statement of the lemma. 
Set $b_0:=b$.  We will occasionally update to a new version of $b$, which will be given by a new index.  The new $b$ will always be an integral power of the previous indexed $b$.

Set $b_1:=b_0^s$.  The element $b_1$ will have rotation number $0/1$ in lowest non-negative terms.  For $b_1$, we have $Fix(b_1)$ is not empty, and also not the whole circle (else $b$ was originally a torsion element in $T$).

\newcommand{\acta}{under the action of $\langle a\rangle$}
\newcommand{\actb}{under the action of $\langle b\rangle$}
\newcommand{\indexI}{\mathscr{I}}

Let $\mathscr{I}\subset S^1$ be such that for each component $C$ of $Supp(b_1)$, we have $|C\cap \indexI|=1$, and associate each such $C$ with its unique point in $\indexI$, so that $\indexI$ becomes an index set for the components of $Supp(b_1)$.  We observe that $\indexI$ comes with an inherent circular order as a subset of $S^1$.  Let $L_b$ represent the set of limit points of $\indexI$ which are not in $\indexI$, and observe that $L_b\subset Fix(b_1)$.

For each positive integer $d$, set $\Delta_d:=[-d,d]\cap(\Z\backslash\{0\})$, the set of non-zero integers a distance $d$ or less from zero.  Now for all positive integers $d$ we can set $\epsilon_d$ to be one half of the distance from $Fix(b_1)$ to the set $\cup _{i\in\Delta_d} a^i\cdot Fix(b_1)$.  Noting that these $\epsilon_d$ are all well defined and non-zero (unless $a$ is torsion) as the sets involved are compact and as $a^m\cdot Fix(b_1)\cap a^n\cdot Fix(b_1) \neq \emptyset$ implies that either $m=n$ or that $a$ is torsion and $n-m$ is divisible by the order of $a$.

 Our analysis now splits, depending on whether or not $a$ is torsion.  In the case that $a$ is torsion, our proof is somewhat easier, so we will execute that proof immediately.

\vspace{.1 in}
{\flushleft {\it \underline{Case}: $a$ is torsion with order $q$.}}

In this case, the value $\epsilon_{q-1}$ explicitly measures one half of the distance between $Fix(b_1)$ and the union of the images of $Fix(b_1)$ under the action of non-trivial powers of $a$.  Set $\mathcal{U}$ to be the open $\epsilon_{q-1}$ neighbourhood of $Fix(b_1)$, and observe that for each integer $i\in \{1,2,\ldots ,q-1\}$ we have $a^i\cdot Fix(b_1)\cap \mathcal{U}=\emptyset$.  For each non-zero $i\in \left\{1,2,\ldots q-1\right\}$ set $\mathcal{U}_i$ to be the $\epsilon_{q-1}$ neighbourhood of $a^i\cdot Fix(b)$.  Again, for all such indices $i$, $\mathcal{U}\cap\mathcal{U}_i=\emptyset$.  Set 
\[
X_b:= \mathcal{U}\bigcap_{1\leq i<q}(a^{q-i}\cdot\mathcal{U}_i). 
\]  Now by construction we have that the image set $a^i\cdot X_b\subset \mathcal{U}_i$, but $\mathcal{U}_i\cap X_b=\emptyset$ since $\mathcal{U}_i\cap \mathcal{U}=\emptyset$ and $X_b\subset \mathcal{U}$.  Thus, $X_b$ is an open set containing $Fix(b_1)$ that is disjoint from its image under the action of any nontrivial power of $a$.

As $X_b$ and the components of support of $b_1$ altogether cover the circle, there is a finite set of open interval components of $Supp(b_1)$ which together with $X_b$ covers the circle.  In turn, this implies there is a minimal positive integer $v$ so that for all $x\in S^1\backslash X_b$, we have $b_1^v x\in X_b$ and $b_1^{-v} x\in X_b$. 

 Now we can set
\[
X_a:= \cup_{1\leq i<q} (a^i\cdot X_b).
\]
With this choice of $X_a$ and $X_b$ we have arranged that we satisfy the hypotheses of Lemma \ref{ping-pong} for the elements $a$ and $a^k$ where $k=v\cdot s$.

{\flushleft {\it \underline{Case}: $a$ is not torsion.}}

Throughout this case, given a set $X\subset S^1$, and $\epsilon>0$, we shall use the notation $N_\epsilon(X)$ to denote the open $\epsilon$-neighbourhood of $X$, that is, all points in $S^1$ a distance less than $\epsilon$ from some point in $X$.

In this case with $a$ not torsion, we must specify the set $F_a:= Fix(a^q)$, which is a closed non-empty subset of the circle which is disjoint from $Fix(b_1)$.  Choose a specific $\epsilon>0$ so that  $N_\epsilon(Fix(b_1))\cap N_\epsilon(F_a)= \emptyset$, noting that such an epsilon value exists as $Fix(b_1)$ and $F_a$ are disjoint compact subsets of $S^1$.

Let $m$ be a positive integer so that both $a^{mq}\cdot Fix(b_1)\subset N_\epsilon(F_a)$ and $a^{-mq}\cdot Fix(b_1)\subset N_\epsilon(F_a)$.  This $m$ exists as $a^{q}$ acts as a monotone strictly increasing, or as a monotone strictly decreasing function over each component of its support, and as the limit point of any point in a component of support of $a^q$ under increasing powers of $a^q$ must be a fixed point of $a^q$ (and similarly under a negative powers of $a^q$), and as $Fix(b_1)$ is a compact set and hence is contained in a union of finitely many components of support of $a^q$.

We now observe that for $n$ an integer with $|n|>m$, we have that $a^{nq}\cdot Fix(b_1)\subset F_a$ as well.  We would like to argue a stronger result now that there is a positive constant $N$ so that for all $j>N$ we have $a^{j}\cdot Fix(b_1)\subset N_\epsilon(F_a)$ and $a^{-j}\cdot Fix(b_1)\subset N_\epsilon(F_a)$.

To make this argument, the main point to observe is that there is an induced action of $\langle a\rangle$ on the set of components of support of $a^q$ which partitions these components into (possibly infinitely many) orbits of size $q$.  Further, as $a$ commutes with $a^q$ and $a$ is orientation preserving, it is easy to see that each such orbit consists of components of support where the action of $a^q$ is increasing on all components of the orbit, or decreasing on all components of the orbit.

It is also the case that there are only finitely many components of support of $a^q$ which are not already wholly contained in $N_\epsilon(F_a)$.  Let $C_1$, $C_2$, $\ldots$, $C_w$ represent these components, and observe that $Fix(b_1)$ is contained in the union $K$ of these compact intervals. For each component $C_j$, let $I_j$ be the closed interval $C_j\backslash N_\epsilon(F_a)$.  Now each of these components $C_j$ are in an orbit of length $q$ amongst the components of support of $a^q$, and in each such orbit the action of $a^q$ on each component is in the same direction.  Hence there is a finite number $N$ so that for all $j>N$ and intervals $I_m$, we have that $a^j\cdot I_m\subset N_\epsilon(F_a)$, and also $a^{-j}\cdot I_m\subset N_\epsilon(F_a)$.

Now define $J$ as below:
\[
J:=\cup_{i\in\Delta_N} ((a^i\cdot Fix(b_1))\cap K).
\]
Where we recall that $\Delta_n=[-n,n]\cap(\Z\backslash\{0\})$ for any particular $n\in\N$.

It is immediate that $J$ is a compact set which is disjoint from $Fix(b_1)$.  As such, there is a $\delta>0$ so that $\delta<\epsilon$ and the $\delta$-neighbourhood $N_\delta(Fix(b_1))$ of $Fix(b_1)$ is disjoint from the set $V_\delta$ defined as
\[
V_\delta:=\cup_{i\in\Delta_N} (a^i\cdot N_\delta(Fix(b_1)))
\]
 and noting that as $\delta<\epsilon$ we also have that $N_\delta(Fix(b_1))$ is disjoint from $N_\epsilon(F_a)$.

Now set $X_b:=N_\delta(Fix(b_1))$ and $X_a:=N_\epsilon(F_a)\cup V_\delta$.

By construction, there is an integer $z>0$ so that $b_1^z$ takes the complement of $X_b$ (and so, $X_a$) into $X_b$, while all non-trivial powers of $a$ take $X_b$ into $X_a$.  Hence the integer $k=s\cdot z$ has the property that  $a$ and $b^k$ freely generate a free group of rank $2$.
\end{proof}

\section{Applying Condition \ref{weak-kesten}, and variants, in $T$}\label{trying-conditions}
Here we list lemmas, where the Condition \ref{weak-kesten} can be used.

\begin{lemma}\label{fixed-point}
Let $H$ be a finite set of nontrivial elements in $T$ so that there is some point $p\in \cap_{h\in H} Supp(h).$   Then, for any positive integer $n$ there is an element $g\in T$  and $\{c_1,c_2,\ldots,c_n\}$ so that $c_i\in C_T(g)$ for all $i$, and so that for all $h\in H$ we have the set
\[
G_h:=\left\{ (gh)^{c_i}\mid i\in \left\{1,2,\ldots, n\right\}\right\}
\]
is a free basis for a free group of rank $n$.
\end{lemma}
\begin{proof}
Let $H$ and $p$ as in the statement of the lemma, and let $n\in \N$ be given.  For each $h\in H$, let $Rot(h):=r_h/s_h$ written in lowest terms (NB, any finite periodic orbit under the action of $\langle h\rangle$ is of length $s_h$).  By the definition of $p$, we see there is a non-empty interval $(a,b)$ with $p\in (a,b)$ so that  for all $h\in H$ we have $(a,b)\cdot h^j\cap (a,b)=\emptyset$ for all $1\leq j<s_h$.

Now let $g\in T$ be an element with rotation number $Rot(g)=0$ which fixes exactly the point $p$.  We can choose $g$ so that for all $h\in H$ the product $gh$ has precisely two fixed points which are generated from the contact of the graph of $g$ with the graph of the element $h^{-1}$, by choosing the graph of $g$ to be near to a step function (the nearly vertical component of the graph of $g$ should be steeper than any slope of any element of $H$, and the nearly horizontal component of the graph of $g$ should have slope closer to zero than any slope of any element of $H$, and, in a small epsilon box around $(p,p)$ which misses the graphs of the elements of $H$, the graph of $g$ is unrestrained).  As $g$ fixes the point $p$, it is the case that that for all elements $h\in H$, we have $k_h\cdot p=gh\cdot p=g\cdot(h\cdot p) \neq p$ and further that under the action of $\langle k_h\rangle$, $p$ is in an infinite orbit which limits to the two ends of the component of support of $k_h$ which contains $p$.
\newcommand{\lcm}{\mathop{LCM}}
In these conditions we can immediately apply Lemma \ref{free-powering-one} to claim that for each $h\in H$, there is a power $g^{\rho_h}$ of $g$ so that $k_h$ and the element $g^{\rho_h}$ freely generate a free group of rank $2$.  Now take $\theta:=\lcm\{\rho_h\mid h\in H\}$, so that $g^\theta$ is free for $k_h$ for each $h$.

Now, it is immediate that setting $c_i:= g^{i\cdot \theta}$ for indices $i$ in the range $1\leq i\leq n$,  we can create sets $X_h:=\{(gh)^{c_i}\}_{i=1}^n$ so that $\langle X_h\rangle$ is free on $n$-generators for all $h\in H$.

\end{proof}
The following is some weakening of Condition \ref{weak-kesten} which we can achieve in Thompson's group $T$.  Unfortunately, by raising to the power $p$, we find a free subgroup of $K:=\left\{ ((hg)^{c_i})\mid i\in \left\{1,2,\ldots, n\right\}\right\}$ which is generally not finite index.

\begin{lemma}\label{powerfree}  Let $H$ be a finite set of nontrivial elements in $T$ with cardinality $p$.  Then there is an element $g\in T$ such that for any positive integer $n$ there are elements  $\{c_1,c_2,\ldots,c_n\}$ so that $c_i\in C_T(g)$ for all $i$, and so that for all $h\in H$  the set
\[
G_h:=\left\{ ((hg)^{c_i})^p\mid i\in \left\{1,2,\ldots, n\right\}\right\}
\]
freely generates a free group of rank $n$.
\end{lemma}
\begin{proof}
Index the $h\in H$ so that $H=\left\{h_1,h_2,\ldots,h_p\right\}$ for some minimal positive integer $p$.  As each element $t\in T$ has a rational rotation number $Rot(t)\in \Q/\Z$, let  $Rot(h_i)=r_i/s_i$ expressed in lowest positive terms.  For each index $i$, inductively choose a point $x_i\in Supp(h_i)$ so that the set $$X_i:=\left\{ h_j^{-1}\cdot x_i, x_i, h_j\cdot x_i\mid 1\leq j\leq i\right\}$$ is disjoint from the union $\cup _{j<i}X_j$, and where each $x_i$ is chosen as a dyadic rational so that no point in $X_i$ is an end of a component of support of $h_j^{s_j}$ for any index $j$.  Note that in all cases we are choosing an $x_i$ so that a finite set of images miss a finite set of points in the circle, and this is easy to do.  Set $P:=\left\{x_i\mid 1\leq i\leq p\right\}$.  Reindex $P$ so that $0\leq x_1<x_2<\cdots<x_p<1$, and reindex the sets $X_i$ correspondingly.

Now choose, for each $i$, an interval $(a_i,b_i)$ centred around $x_i$, and set $$I_i:=\left\{h_j^{-1}\cdot(a_i,b_i), (a_i,b_i), h_j\cdot (a_i,b_i)\mid 1\leq j\leq i\right\}$$ where each $(a_i,b_i)$ is chosen small enough so that each element of $I_i$ is disjoint from the union $\cup_{j<i, B\in I_j}B$, and where $(a_i,b_i)$ is disjoint from $h_j^{-1}\cdot(a_i,b_i)$  and $h_j\cdot(a_i,b_i)$ whenever $x_i\in Supp(h_j)$, and where $h_j^{-1}\cdot(a_i,b_i)$  and $h_j\cdot(a_i,b_i)$ intersect each other non-trivially only if $h_j$ is torsion of order two.  We can insist these intervals are possibly smaller still, so that the complement of the union $\tilde{U}:=\cup_{1\leq j\leq p, B\in I_j}B$  is a union of closed intervals each of which has non-empty interior.  We observe that by our choices of the sets $X_i$, we can so choose our intervals $(a_i,b_i)$.

Now let $\tilde{g}$ be a non-torsion element with rotation number $1/p$ which admits exactly one orbit of size $p$ under the action of $\langle\tilde{g}\rangle$, where the orbit is the set $P$, and where $\tilde{g}$ cyclically permutes the $x_i$ in order, taking $x_1$ to $x_2$ and etc. 

The element $\gamma:=\tilde{g}^p$ has $p$ components of support $S_1:=(x_1,x_2)$, $S_2:=(x_2,x_3)$, $\dots$, $S_p:=(x_p,x_1)$, and $\gamma$ is either increasing on each component of support, or decreasing on each component of support.  We re-choose $\tilde{g}$ if necessary so that it satisfies all previous conditions and so that $\gamma$ is increasing on each component of support, and note that such elements $\tilde{g}$ exist.

Now set $U:=\cup_{1\leq j\leq p} (a_i,b_i)$, which union also has the property that the complement $C:= S^1\backslash U$ is a union of $p$ disjoint closed intervals $D_i:=[b_i,a_{i+1}]$ (where we set $a_{p+1}:= a_1$) each of which is not a point.  As such is the case, we can find a positive integer $d$ so that $\gamma^d\cdot C \subset U$ as after $d$ iterations of $\gamma$, the set $D_i$ will be moved so as to have image just to the left of $x_{i+1}$ (where again, we set $x_{p+1}:= x_1$).

Now set $g:= \gamma^d\tilde{g}$.  The element $g$ so constructed has rotation number $1/p$ as before, and acts on the set $P$ as $\tilde{g}$ does, but it has the further property that for all integers $s\neq 0$, $g^s\cdot C\subset U.$

Now set, for each index $i$, $k_i:=g h_i$.  We will now show that for all $k_i$, the point $x_j$ is not in any finite periodic orbit of $k_i$, for any index $j$.  

Let $i$ and $j$ be indices so that $k_i$ and $x_j$ are an element and point as defined above, respectively.  Let $q$ be an index and consider the interval $(a_q,x_q]$ under the action of $g$, and then consider the interval $(a_j,x_j]$ under the action of $k_i$.

Under the action of $g$, we know already that $x_q\mapsto x_{q+1}$, and that the component $[b_{q-1},a_q]$ of $C$ maps into the interval $(a_{q+1},x_{q+1})$.  In particular, the whole interval $[b_{q-1},x_q]$ is mapped into $(a_{q+1},x_{q+1}]$ by $g$ and so the interval $[a_q,x_q]$ is mapped into $(a_{q+1},x_{q+1}]$ as well, and precisely the point $x_q$ will map to $x_{q+1}$.

Now consider the action of $k_i$ on $(a_j,x_j]$.  If $x_j$ is in the support of $h_i$, then $h_i\cdot(a_j,x_j]\subset C$, whereupon that image is moved into some interval $(a_*,x_*)$ by the action of $g$.

In particular, $k_i^z\cdot x_j = x_{j+z}$ (in cyclic order) as long as for all $0\leq y\leq z$ we have $x_{j+y}\not\in Supp(h_i)$.  If instead for some minimal integer $s\in\left\{0,1,\ldots,p-1\right\}$ we have $x_{j+s}\in Supp (h_i)$ then on $k_i^{s+1}$ will move $x_j$ out of $P$ as we will have $k_i^{s+1}\cdot x_j\in (a_{*},x_{*})$ for some index $*$.  Furthermore, if this happens, all further iterates of $x_j$ will fail to re-enter $P$.

However, $x_i$ is itself in the support of $h_i$, so this eventuality happens (in $p$ steps or fewer), and so no $x_j$ is on a finite periodic orbit under the action of $\langle k_i\rangle$, for any indices $j$ and $i$.

In particular, for all indices $j$, we see that $ (k_j^{tp}\cdot P)\cap P=\emptyset$ for all $t\in \Z\backslash\left\{0\right\}$. 

We can now quote Lemma \ref{free-powering-one}, using $a:= k_j^p$ for each index $j$ and $b=g$ to claim that for each index $j$ there is an integer $z_j$ so that the group $\langle k_i^p, g^{z_j}\rangle$ is free on two generators  (in this setup, $P=Fix(b^s)=Fix(g^p)$).  Now set $z:=LCM(\left\{z_1,z_2,\ldots,z_p\right\})$ the least common multiple of the values $z_j$, and then we have that for all indices $j$, the elements $k_j^p$ and $g^z$ generate a free group.  In particular, setting $c_i:= g^{iz}$ for $i\in \left\{1,2,\ldots, n\right\}$ we see that $\Gamma_j:=\left\{ (k_j^p)^{c_i}=(k_j^{c_i})^p\mid 1\leq i\leq n, i\in\Z\right\}$ freely generates a free group of rank $n$.

But now, recalling that $k_i=g h_i$, and that therefore $[g,c_i]=1$, we have proved our claim. 
\end{proof}

\section{Some questions}\label{conjectures}
The last result in particular shows that our Condition \ref{weak-kesten} is almost (in some sense) enough to show that $\cstar{T}$ is simple.  But not quite! This, together with our partial converse of the Haagerup-Olesen result,  encourages us to ask the foliowing question.
\begin{question}
Is the non-amenability of Thompson's group $F$ equivalent to the simplicity of the algebra $\cstar{T}$.
\end{question}

As the setup of Condition \ref{weak-kesten} is more flexible than that of Condition \ref{kesten-test}, we find Condition \ref{weak-kesten} easier to use.  Still, we have not actually proven that the conditions are not equivalent.  Thus it would be quite useful to give a positive answer to the following question.
\begin{question}
Is there a group $G$ which fails to satisfy Condition \ref{kesten-test} but which does satisfy Condition \ref{weak-kesten}?
\end{question}

In \cite{JM}, \cite{JS} and \cite{JNS} several approaches to amenability via group actions were developed. It is an interesting question to relate these approaches to properties of a group's $C^*$-algebra.

\bibliographystyle{amsplain}
\bibliography{ploiBib.bbl}
\end{document}